\documentclass[11pt]{amsart}

\usepackage{amscd}
\usepackage{amsmath, amssymb}
\usepackage{amsfonts}
\usepackage{enumerate}
\newcommand{\de}{\partial}

\newcommand{\dbar}{\overline{\partial}}

\newcommand{\ddbar}{\sqrt{-1} \partial \overline{\partial}}

\newcommand{\ov}[1]{\overline{#1}}

\newcommand{\ti}[1]{\tilde{#1}}
\newcommand{\vp}{\varphi}
\newcommand{\vol}{\mathrm{Vol}}

\renewcommand{\leq}{\leqslant}
\renewcommand{\geq}{\geqslant}
\renewcommand{\le}{\leqslant}
\renewcommand{\ge}{\geqslant}

\newcommand{\be}{\begin{equation}}
\newcommand{\ee}{\end{equation}}

\newcommand{\ol}{\overline}
\begin{document}
\newtheorem{claim}{Claim}
\newtheorem{theorem}{Theorem}[section]
\newtheorem{lemma}[theorem]{Lemma}
\newtheorem{corollary}[theorem]{Corollary}
\newtheorem{proposition}[theorem]{Proposition}
\newtheorem{question}{question}[section]
\theoremstyle{definition}
\newtheorem{remark}[theorem]{Remark}
\newtheorem{conjecture}[theorem]{Conjecture}

\numberwithin{equation}{section}

\title{The Fu-Yau equation in higher dimensions}
\dedicatory{Dedicated to Professor Gang Tian on the occasion of his 60th birthday}
\author[J. Chu]{Jianchun Chu}
\address{Institute of Mathematics, Academy of Mathematics and Systems Science, Chinese Academy of Sciences, Beijing 100190, P. R. China}
\email{chujianchun@gmail.com}
\author[L. Huang]{Liding Huang}
\address{School of Mathematical Sciences, University of Science and Technology of China, Hefei 230026, P. R. China}
\email{huangld@mail.ustc.edu.cn}
\author[X. Zhu]{Xiaohua Zhu}
\address{School of Mathematical Sciences, Peking University, Yiheyuan Road 5, Beijing 100871, P. R. China }
\email{xhzhu@math.pku.edu.cn}



\begin{abstract}
In this paper, we prove the existence of solutions to the Fu-Yau equation on compact K\"{a}hler manifolds. As an application, we give a class of non-trivial solutions of the modified Strominger system.
\end{abstract}

\maketitle

\section{Introduction}
In 1985, Strominger proposed a new system of equations, now referred as the Strominger system, on $3$-dimensional complex manifolds \cite{Str86}. This system arises from the study on supergravity in theoretical physics.
Mathematically, the Strominger system can be regarded as a generalization of the Calabi equation for Ricci-flat K\"ahler metrics to
non-K\"{a}hler spaces \cite{YS05}. It is also related to Reid's fantasy on the moduli space of Calabi-Yau threefolds \cite{RM87}.

Let us first recall this system. Assume that $X$ is a $3$-dimensional Hermitian manifold which admits a nowhere vanishing holomorphic $(3,0)$-form $\Omega$. Let $E\rightarrow X$ be a holomorphic vector bundle with Hermitian metric $H$.  The Strominger system is given by
\begin{equation}\label{Hermitian-Einstein equation}
F_{H}\wedge\omega_{X}^{2} = 0, F_{H}^{2,0}  = F_{H}^{0,2} = 0;
\end{equation}
\begin{equation}\label{Dilation equation}
d^{*}\omega_{X} = \sqrt{-1}(\ov{\partial}-  \partial)\log\|\Omega\|_{\omega_{X}};
\end{equation}
\begin{equation}\label{Bianchi identity}
\ddbar\omega_{X}-\frac{\alpha}{4}(\textrm{tr}R\wedge  R  -\textrm{tr}F_{H}\wedge F_{H})=0,
\end{equation}
where $\omega_X$ is a Hermitian metric on $X$ with the Chern curvature $R$ and $F_{H}$ is the curvature of Hermitian metrics $(E,H)$. By $\textrm{tr}$, we denote
the trace of Endomorphism bundle of either $E$ or $TX$.

To achieve a supersymmetry theory, both $H$ and $\omega_{X}$ have to satisfy \eqref{Hermitian-Einstein equation} and \eqref{Dilation equation}.
\eqref{Dilation equation} is also called the dilation equation.
Li and Yau observed that it is equivalent to a conformally balanced condition \cite{LY05},
\begin{equation*}
d(\|\Omega\|_{\omega_{X}}\omega_{X}^{2}) \,=\, 0.
\end{equation*}
The lest understood equation of the system  is \eqref{Bianchi identity} known as the Bianchi identity, which is also related to index theorey for Dirac operators (\cite{FD86}, \cite{WE85}), the topological theory of string structures (\cite{BU11}, \cite{RC11}, \cite{SSS12}) and generalized geometry (\cite{BH15}, \cite{FM14}, \cite{FRT}).
It is an equation on $4$-forms and intertwines $\omega_{X}$ with the curvatures $R$ and $F_{H}$, which is very difficult to understand in view of analysis.

We know little about the Strominger in general except for a few special spaces on which one can make use of particular structures.
In \cite{LY05}, Li and Yau found the first irreducible smooth solution. They considered a stable holomorphic bundle $E$ of rank $r=4,5$ on a Calabi-Yau $3$-fold $X$
and constructed a solution of the Strominger system as a perturbation of a Calabi-Yau metric on $X$ and a Hermitian-Einstein metric $H$ on $E$.

In \cite{FuY08}, Fu and Yau constructed non-perturbative, non-K\"{a}hler solutions of the Strominger system on a toric fibration over a $K3$ surface constructed by Goldstein-Prokushki. Let us recall this construction. Let
$\omega_{1}$ and $\omega_{2}$ be two anti-self-dual $(1,1)$ forms on a $K3$ surface $(S,\omega_{S})$  $(\omega_{S}$ is a K\"{a}hler-Ricci flat metric on $S$)  with a nowhere vanishing holomorphic $(2,0)$-form $\Omega_{S}$  satisfying: $[\frac{\omega_{1}}{2\pi}],[\frac{\omega_{2}}{2\pi}]\in H^{1,1}(S,\mathbb{Z})$. In
\cite{GP04}, Goldstein-Prokushki constructed a toric fibration $\pi: X\rightarrow S$ which is determined by $\omega_{1}$, $\omega_{2}$ and a $(1,0)$ form $\theta$ on $X$ such that
\begin{equation*}
\Omega = \pi^{*}(\Omega_{S})\wedge\theta
\end{equation*}
defines a nowhere vanishing holomorphic $(3,0)$ form on $X$.  Then, for any $\vp\in C^{\infty}(S)$, $(X,\omega_{\vp})$ always satisfies (\ref{Dilation equation}), where
\begin{equation*}
\omega_{\vp} = \pi^{*}(e^{\vp}\omega_{S})+\sqrt{-1}\theta\wedge\ov{\theta}.
\end{equation*}
Thus if $E\rightarrow X$ is a degree zero stable holomorphic vector bundle with a Hermitian-Einstein metric $H$ on $E$,  $(\pi^{*}E,\pi^{*}H,X,\omega_{\vp})$ satisfies both (\ref{Hermitian-Einstein equation}) and (\ref{Dilation equation}). In \cite{FuY07},  Fu and Yau showed that (\ref{Bianchi identity}) for   $(\pi^{*}E,\pi^{*}H,X,\omega_{\vp})$  is equivalent to the following equation for $\varphi$, also called the Fu-Yau equation,
\begin{equation}\label{Fu-Yau equation-1}
\ddbar(e^{\vp}\omega_{S}-\alpha e^{-\vp}\rho)+2\alpha\ddbar\vp\wedge\ddbar\vp+\mu\frac{\omega_{S}^{2}}{2!}=0,
\end{equation}
where $\rho$ is a real-valued smooth $(1,1)$-form, $\mu$ is a smooth function and $\alpha\neq0$ is a constant called slope parameter.
They further proved the existence for (\ref{Fu-Yau equation-1}) in the case of $\alpha<0$ \cite{FuY08} and $\alpha>0$ \cite{FuY07} on K\"{a}hler surfaces, respectively.

In higher dimensions, Fu and Yau proposed a modified Strominger system for $(E, H, X, \omega_X )$,
\begin{equation*}
F_{H}\wedge\omega_{X}^{n} = 0,\quad F_{H}^{2,0} = F_{H}^{0,2} = 0;
\end{equation*}
\begin{equation*}\label{Modified dilatino equation}
d(\|\Omega\|_{\omega_{X}}^{\frac{2(n-1)}{n}}\omega_{X}^{n}) = 0;
\end{equation*}
\begin{equation*}\label{Modified Bianchi identity}
\left(\ddbar\omega_{X}-\frac{\alpha}{4}(\textrm{tr}R\wedge R-\textrm{tr}F_{H}\wedge F_{H})\right)\wedge\omega_{X}^{n-2} = 0.
\end{equation*}
Here $X$ is an $(n+1)$-dimensional Hermitian manifold, equipped with a nowhere vanishing holomorphic $(n+1,0)$ form $\Omega$.  Clearly, the modified Strominger system is the same as the original Strominger system when $n=2$. Given any Calabi-Yau manifold $M$
with a nowhere vanishing holomorphic $(n,0)$ form $\Omega_{M}$, Goldstein-Prokushki's construction gives rise to a toric fibration $\pi:X\mapsto M$ as in case of $K3$ surfaces. Fu and Yau showed that the modified Strominger system for $(\pi^{*}E,\pi^{*}H,X,\omega_{\vp})$ can be reduced to the Fu-Yau equation on $M$,
\begin{equation}\label{Fu-Yau equation}
\begin{split}
\ddbar(e^{\vp}\omega & -\alpha e^{-\vp}\rho) \wedge\omega^{n-2} \\
& +n\alpha\ddbar\vp\wedge\ddbar\vp\wedge\omega^{n-2}+\mu\frac{\omega^{n}}{n!}=0.
\end{split}
\end{equation}

More recently, Phong, Picard and Zhang proved the existence for (\ref{Fu-Yau equation}) in higher dimensions when $\alpha<0$ \cite{PPZ16b}. However,
the solvability of (\ref{Fu-Yau equation}) in higher dimensions is still open when $\alpha>0$.
The purpose of present paper is to give a complete solution in this case. Actually, we will give a unified way for (\ref{Fu-Yau equation}) in higher dimensions in
both cases $\alpha>0$ and $\alpha<0$, more precisely, we prove

\begin{theorem}\label{Existence Theorem}
Let $(M,\omega)$ be an $n$-dimensional compact K\"{a}hler manifold. There exists a small constant $A_0>0$ depending only on $\alpha$, $\rho$, $\mu$ and $(M,\omega)$ such that for any positive $A\leq A_0$, there exists a smooth solution $\vp$ of (\ref{Fu-Yau equation}) satisfying the elliptic condition
\begin{equation}\label{Elliptic condition}
\tilde{\omega} = e^{\vp}\omega+\alpha e^{-\vp}\rho+2n\alpha\ddbar\vp \in \Gamma_{2}(M),
\end{equation}
and the normalization condition
\begin{equation}\label{Normalization condition}
\|e^{-\vp}\|_{L^{1}} = A,
\end{equation}
where $\Gamma_{2}(M)$ is the space of $2$-th convex $(1,1)$-forms (cf. Section 3).
\end{theorem}

\begin{remark}
We point out that if $\alpha<0$ and $n=2$, our normalization condition (\ref{Normalization condition}) is the same as that in \cite{FuY07}. However, in the case that $\alpha>0$ and $n=2$, Fu and Yau \cite{FuY08} solved (\ref{Fu-Yau equation}) under the normalization condition $\|e^{-\vp}\|_{L^{4}}=A$, which is stronger than (\ref{Normalization condition}). When $\alpha<0$ and $n>2$, Phong, Picard and Zhang used a different normalization condition $\|e^{\vp}\|_{L^{1}}=\frac{1}{A}$. Hence, our result is also new compared to the results cited above.
\end{remark}

As a geometric application of Theorem \ref{Existence Theorem}, we prove

\begin{theorem}
For any $n\geq 2$, there exists a function $\vp\in C^{\infty}(M)$ such that the Fu-Yau's reduction $(\pi^{*}E,\pi^{*}H,X,\omega_{\vp})$ yields a smooth solution of the modified Strominger system.
\end{theorem}

From the view point of PDE, (\ref{Fu-Yau equation}) can be written as a $2$-nd Hessian equation of the form
\begin{align}\label{fu-2}
\sigma_{2}(\ti{\omega})=F(z,\varphi,  \de\vp),
\end{align}
where
$$F(z,\varphi,  \de\vp)= \frac{n(n-1)}{2}\left(e^{2\vp}-4\alpha e^\varphi|\de\vp|_{g}^{2}\right)+\frac{n(n-1)}{2}f(z,\vp,\de\vp)$$
and $f(z,\vp,\de\vp)$  satisfies (cf. (\ref{Definition of f})),
$$|f(z,\vp,\de\vp)|\le C(e^{-2\varphi}+ e^{-\varphi}|\nabla\varphi|^2+1).$$

There are many interesting works for the $k$-th complex Hessian equation of the form:
\begin{align}\label{fu-2-2}
\sigma_{k}(\omega+\ddbar\vp)\,=\,F(z).
\end{align}
For examples, Hou, Ma and Wu proved the second order estimate for \eqref{fu-2-2} \cite{HMW10}; Combining Hou-Ma-Wu's estimate with a blow-up argument,
Dinew and Ko{\l}odziej solved \eqref{fu-2-2} \cite{DK12};  Sz\'ekelyhidi, and also Zhang, obtained analogous result in the Hermitian case \cite{Sze15}, \cite{Zha15}.

However, for the Fu-Yau equation (\ref{Fu-Yau equation}), new difficulties arise because the right hand side $F$ of (\ref{fu-2}) depends on $\de\vp$.  Moreover,
(\ref{Fu-Yau equation}) may become degenerate when $\alpha>0$. This makes a big difference between the case $\alpha>0$ and the case $\alpha <0$. When $\alpha<0$,   there is no issue on non-degeneracy. However,  when $\alpha >0$, one needs to establish a non-degeneracy estimate. In dimension $2$, Fu and Yau obtained such an estimate \cite{FuY08}. Unfortunately, their arguments do not work in higher dimensions. It has been a main obstacle to solving (\ref{Fu-Yau equation}) in higher dimensions when $\alpha >0$.

In this paper, we find a new method for establishing the non-degeneracy estimate. This estimate is different from either Fu-Yau's one in \cite{FuY07,FuY08} or   Phong-Picard-Zhang's one in \cite{PPZ15,PPZ16b}. We regard the first and second order estimates as a whole and derive the required non-degeneracy estimate.
To be more specific, assuming that
\begin{equation*}
|\de\dbar\vp|_{g} \leq D_{0},
\end{equation*}
where $D_{0}$ is a  constant (depending only on $n$, $\alpha$, $\rho$, $\mu$ and $(M,\omega)$) to be determined later, we derive a stronger gradient estimate ( independent of $D_0$)   by choosing a small number $A$ in (\ref{Normalization condition})  (cf. Proposition \ref{Gradient estimate}). Then using this stronger gradient estimate,   we obtain an improved estimate (cf. Proposition \ref{Second order estimate}),
\begin{equation*}
|\de\dbar\vp|_{g} \leq \frac{D_{0}}{2}.
\end{equation*}
This can be used to obtain  an a prior $C^2$-estimate and consequently  the  non-degeneracy estimate   via the continuity method (cf. (\ref{non-degenrate}) in Section 5).

From the proof of Theorem \ref{Existence Theorem},  we also prove the  following uniqueness  result  of  (\ref{Fu-Yau equation}).

\begin{theorem}\label{Uniqueness Theorem}The solution $\vp$ of (\ref{Fu-Yau equation}) is unique  if it
satisfies  (\ref{Elliptic condition}), (\ref{Normalization condition}) and
\begin{equation}\label{condition-uniqueness}
e^{-\vp} \leq \delta_{0},~~|\de\dbar\vp|_{g} \leq D,~~D_{0}\leq D \text{~and~} A \leq \frac{1}{C_{0}M_{0}D},
\end{equation}
where  $C_0$ is a uniform constant, and  $\delta_{0}, M_0$  and  $D_{0}$  are   constants determined    in Proposition \ref{Zero order estimate}, Proposition \ref{Second order estimate}, respectively.
\end{theorem}

Since the normalization condition (\ref{Normalization condition}) is different from ones in the previous works as in \cite{FuY07,FuY08, PPZ15,PPZ16b}, etc., we shall also derive $C^0, C^1, C^2$-estimates for solutions of (\ref{fu-2}) step by step.

The  paper is organized  for each estimate in  one section.  Theorem  \ref{Existence Theorem} and Theorem  \ref{Uniqueness Theorem} are both proved  in last section, Section 5.

\noindent {\bf Acknowledgements.} On the occasion of his 60th birthday,  the authors  would like to thank Professor Gang Tian for his  guidance   and  encouragement   in mathematics.   His insight and teaching in mathematics give us a lot of benefits in  past years.    It is our pleasure to dedicate this paper to him.

\section{Zero order estimate}
In this section, we  use the iteration method to derive  the following  zero order estimate of $\vp$ to (\ref{Fu-Yau equation}).

\begin{proposition}\label{Zero order estimate}
Let $\vp$ be a smooth solution of (\ref{Fu-Yau equation}). There exist constants $A_{0}$ and $M_{0}$ depending only on $\alpha$, $\rho$, $\mu$ and $(M,\omega)$ such that if
\begin{equation*}
e^{-\vp} \leq \delta_{0} := \sqrt{\frac{1}{2|\alpha|\|\rho\|_{C^{0}}+1}} \text{~and~} \|e^{-\vp}\|_{L^{1}} = A \leq A_{0},
\end{equation*}
then
\begin{equation}\label{c0-estimate}
\frac{1}{M_{0}A}\leq e^{\inf_{M}\vp} \text{~and~} e^{\sup_{M}\vp}\leq \frac{M_{0}}{A}.
\end{equation}
\end{proposition}

\begin{proof} We first do  the infimum estimate. The supremum estimate depends on the established  infimum estimate.
By the choice  of $\delta_{0}$ and  the condition $e^{-\vp}\leq\delta_{0}$, it is clear that
\begin{equation}\label{Infimum estimate equation 4}
\omega+\alpha e^{-2\vp}\rho\geq\frac{1}{2}\omega.
\end{equation}
By the elliptic condition (\ref{Elliptic condition}), we have for $k\geq 2$,
\begin{equation*}
k\int_{M}e^{-k\vp}\sqrt{-1}\de\vp\wedge\dbar\vp\wedge\ti{\omega}\wedge\omega^{n-2} \geq 0.
\end{equation*}
By    the Stokes' formula,   it follows  that
\begin{equation}\label{Infimum estimate equation 2}
\begin{split}
& -k\int_{M}e^{-k\vp}(e^{\vp}\omega+\alpha e^{-\vp}\rho)\wedge\sqrt{-1}\de\vp\wedge\dbar\vp\wedge\omega^{n-2} \\
\leq {} & 2n\alpha k\int_{M}e^{-k\vp}\sqrt{-1}\de\vp\wedge\dbar\vp\wedge\ddbar\vp\wedge\omega^{n-2} \\
   = {} & -2n\alpha \int_{M}\sqrt{-1}\de e^{-k\vp}\wedge\dbar\vp\wedge\ddbar\vp\wedge\omega^{n-2} \\
   = {} & 2n\alpha \int_{M}e^{-k\vp}\ddbar\vp\wedge\ddbar\vp\wedge\omega^{n-2} \\
   = {} & -2\int_{M}e^{-k\vp}\ddbar(e^{\vp}\omega-\alpha e^{-\vp}\rho)\wedge\omega^{n-2}-2\int_{M}e^{-k\vp}\mu\frac{\omega^{n}}{n!}.
\end{split}
\end{equation}
In the last equality, we used the equation (\ref{Fu-Yau equation}).

For the first term of  right hand side  in (\ref{Infimum estimate equation 2}),   we compute
\begin{equation}\label{Infimum estimate equation 5}
\begin{split}
& -2\int_{M}e^{-k\vp}\ddbar(e^{\vp}\omega-\alpha e^{-\vp}\rho)\wedge\omega^{n-2} \\
= {} & -2k\int_{M}e^{-k\vp}\sqrt{-1}\de\vp\wedge\dbar(e^{\vp}\omega-\alpha e^{-\vp}\rho)\wedge\omega^{n-2} \\
= {} & -2k\int_{M}e^{-k\vp}(e^{\vp}\omega+\alpha e^{-\vp}\rho)\wedge\sqrt{-1}\de\vp\wedge\dbar\vp\wedge\omega^{n-2} \\
     & +2\alpha k\int_{M}e^{-(k+1)\vp}\sqrt{-1}\de\vp \wedge\dbar\rho\wedge\omega^{n-2}.
\end{split}
\end{equation}
Substituting (\ref{Infimum estimate equation 5}) into (\ref{Infimum estimate equation 2}),  we see that
\begin{equation}\label{Infimum estimate equation 3}
\begin{split}
& k\int_{M}e^{-k\vp}(e^{\vp}\omega+\alpha e^{-\vp}\rho)\wedge\sqrt{-1}\de\vp\wedge\dbar\vp\wedge\omega^{n-2} \\
\leq {} & 2\alpha k\int_{M}e^{-(k+1)\vp}\sqrt{-1}\de\vp \wedge\dbar\rho\wedge\omega^{n-2}-2\int_{M}e^{-k\vp}\mu\frac{\omega^{n}}{n!}.
\end{split}
\end{equation}
Combining (\ref{Infimum estimate equation 3}) with (\ref{Infimum estimate equation 4}) and the Cauchy-Schwarz inequality, it follows that
\begin{equation*}
\begin{split}
k\int_{M}e^{-(k-1)\vp}|\de\vp|_{g}^{2}\omega^{n}
\leq {} & Ck\int_{M}\left(e^{-(k+1)\vp}|\de\vp|_{g}+e^{-k\vp}\right)\omega^{n} \\
\leq {} & \frac{k}{2}\int_{M}e^{-(k-1)\vp}|\de\vp|_{g}^{2}\omega^{n}+Ck\int_{M}\left(e^{-(k+3)\vp}+e^{-k\vp}\right)\omega^{n}.
\end{split}
\end{equation*}
Recalling $e^{-\vp}\leq\delta_{0}$, we get
\begin{equation*}
\frac{k}{2}\int_{M}e^{-(k-1)\vp}|\de\vp|_{g}^{2}\omega^{n}
\leq Ck(\delta_{0}^{4}+\delta_{0})\int_{M}e^{-(k-1)\vp}\omega^{n},
\end{equation*}
which implies
\begin{equation*}
\int_{M}|\de e^{-\frac{(k-1)\vp}{2}}|_{g}^{2}\omega^{n} \leq Ck^{2}\int_{M}e^{-(k-1)\vp}\omega^{n}.
\end{equation*}
Replacing $k-1$ by $k$, for $k\geq 1$, we deduce
\begin{equation*}
\int_{M}|\de e^{-\frac{k\vp}{2}}|_{g}^{2}\omega^{n} \leq Ck^{2}\int_{M}e^{-k\vp}\omega^{n}.
\end{equation*}
Hence, by  the Moser iteration together with (\ref{Normalization condition}),  we obtain
$$\|e^{-\vp}\|_{L^{\infty}}\leq C\|e^{-\vp}\|_{L^{1}} = CA.$$
As a consequence, we prove
\begin{equation}\label{infimum estimate}\frac{1}{M_{0}A}\leq e^{\inf_{M}\vp}.
\end{equation}

Next we do  the supremum estimate.
By the similar calculation of (\ref{Infimum estimate equation 2})-(\ref{Infimum estimate equation 3}), for $k\geq1$, we have
\begin{equation*}
\begin{split}
        & k\int_{M}e^{k\vp}(e^{\vp}\omega+\alpha e^{-\vp}\rho)\wedge\sqrt{-1}\de\vp\wedge\dbar\vp\wedge\omega^{n-2} \\
\leq {} & 2\alpha k\int_{M}e^{(k-1)\vp}\sqrt{-1}\de\vp \wedge\dbar\rho\wedge\omega^{n-2}+2\int_{M}e^{k\vp}\mu\frac{\omega^{n}}{n!}.
\end{split}
\end{equation*}
Combining this with (\ref{Infimum estimate equation 4}), we have
\begin{equation*}
\int_{M}e^{(k+1)\vp}|\de\vp|_{g}^{2}\omega^{n} \leq C\int_{M}\left(e^{(k-1)\vp}|\de\vp|_{g}+e^{k\vp}\right)\omega^{n}.
\end{equation*}
Using $e^{-\vp}\leq\delta_{0}$ and the Cauchy-Schwarz inequality, it then follows that
\begin{equation}\label{Supremum estimate equation 2}
\int_{M}e^{(k+1)\vp}|\de\vp|_{g}^{2}\omega^{n} \leq C\int_{M}e^{k\vp}\omega^{n}.
\end{equation}
Moreover,  by \eqref{infimum estimate}, we get
\begin{equation}\label{Supremum estimate equation 1}
\int_{M}e^{k\vp}|\de\vp|_{g}^{2}\omega^{n} \leq C\int_{M}e^{k\vp}\omega^{n}.
\end{equation}
We will use (\ref{Supremum estimate equation 1}) to do the iteration.  We need

  \begin{claim}\label{l1-estimate}
   \begin{align}\label{claim} \|e^{\vp}\|_{L^{1}}\leq\frac{C}{A}.
   \end{align}
   \end{claim}

  Without loss of generality, we assume that $\vol(M,\omega)=1$.
We define a  set by
\begin{equation*}
U = \{x\in M~|~e^{-\vp(x)}\geq\frac{A}{2}\}.
\end{equation*}
Then by (\ref{infimum estimate}) and  (\ref{Normalization condition}),   we have
\begin{equation*}
\begin{split}
A =    {} & \int_{M}e^{-\vp}\omega^{n} \\
  =    {} & \int_{U}e^{-\vp}\omega^{n}+\int_{M\setminus U}e^{-\vp}\omega^{n} \\
  \leq {} & e^{-\inf_{M}\vp}\vol(U)+\frac{A}{2}(1-\vol(U)) \\
  \leq {} & \left(M_{0}-\frac{1}{2}\right)A\vol(U)+\frac{A}{2}.
\end{split}
\end{equation*}
It implies
\begin{equation}\label{Supremum estimate equation 3}
\vol(U) \geq \frac{1}{C_{0}}.
\end{equation}

On the other hand,  by  the Poincar\'{e} inequality and (\ref{Supremum estimate equation 2}) (taking $k=1$), we have
\begin{equation}
\int_{M}e^{2\vp}\omega^{n}-\left(\int_{M}e^{\vp}\omega^{n}\right)^{2}
\leq C\int_{M}|\de e^{\vp}|_{g}^{2}\omega^{n}
\leq C\int_{M}e^{\vp}\omega^{n}.\notag
\end{equation}
By (\ref{Supremum estimate equation 3}) and the Cauchy-Schwarz inequality, we obtain
\begin{equation*}
\begin{split}
\left(\int_{M}e^{\vp}\omega^{n}\right)^{2}
& \leq (1+C_{0})\left(\int_{U}e^{\vp}\omega^{n}\right)^{2}
       +\left(1+\frac{1}{C_{0}}\right)\left(\int_{M\setminus U}e^{\vp}\omega^{n}\right)^{2}\\
& \leq \frac{4(1+C_{0})}{A^{2}}(\vol(U))^{2}+\left(1+\frac{1}{C_{0}}\right)(1-\vol(U))^{2}\int_{M}e^{2\vp}\omega^{n}\\
& \leq \frac{4(1+C_{0})}{A^{2}}+\left(1-\frac{1}{C_{0}^{2}}\right)\left(\left(\int_{M}e^{\vp}\omega^{n}\right)^{2}
       +C\int_{M}e^{\vp}\omega^{n}\right).
\end{split}
\end{equation*}
Clearly,  the above  implies  (\ref{claim}).

By Claim \ref{l1-estimate}, (\ref{Supremum estimate equation 1}) and the Moser iteration,  we see that
\begin{equation*}
\|e^{\vp}\|_{L^{\infty}}\leq C\|e^{\vp}\|_{L^{1}}\leq\frac{C}{A}.
\end{equation*}
Thus
$$\|e^{\vp}\|_{L^{\infty}}\leq\frac{C}{A}.$$

\end{proof}

\section{First order estimate}
In this section, we give  the first order estimate of $\varphi$. For convenience, in this and next section, we say a constant  is uniform if it depends only on $\alpha$, $\rho$, $\mu$ and $(M,\omega)$.

\begin{proposition}\label{Gradient estimate}
Let $\vp$ be a solution of (\ref{Fu-Yau equation}) satisfying (\ref{Elliptic condition}).
Assume that
\begin{equation*}
\frac{A}{M_0}\leq e^{-\vp}\leq M_0A  \text{~and~} |\de\dbar\vp|_{g}\leq D,
\end{equation*}
where $M_0$ is a uniform constant.  Then there exists a uniform constant $C_{0}$ such that if
\begin{equation*}
A \leq A_{D} := \frac{1}{C_{0}M_{0}D},
\end{equation*}
then
\begin{equation*}
|\de\vp|_{g}^{2} \leq M_{1},
\end{equation*}
where $M_{1}$ is a uniform constant.
\end{proposition}

\begin{remark}
The key point in Proposition \ref{Gradient estimate} is that $M_{1}$ is independent of $D$.
The constant $D$ can be chosen arbitrary  large  and  the constant  $A_{D}$ depends on $D$.  This will play an important role in the second order estimate  next section.   In fact, we will  determine $D$ so that $A$ can be determined (cf.  Proposition  \ref{Second order estimate}).
\end{remark}

As usually,  for any $\eta=(\eta_{1},\eta_{2},\cdots,\eta_{n})\in\mathbb{R}^{n}$,  we define
\begin{equation*}
\begin{split}
& \sigma_{k}(\eta) =\sum_{1<i_{1}<\cdots<i_{k}<n}\eta_{i_{1}}\eta_{i_{2}}\cdots\eta_{i_k},\\
\Gamma_{2}  = & \{ \eta\in\mathbb{R}^{n}~|~ \text{$\sigma_{j}(\eta)>0$ for $j=1,2$} \}.
\end{split}
\end{equation*}
Clearly $\sigma_2$ is a $2$-multiple functional.  Then one can extend it to $A^{1,1}(M)$ by
\begin{equation*}
\sigma_{k}(\alpha)=\left(
\begin{matrix}
n\\
k
\end{matrix}
\right)
\frac{\alpha^{k}\wedge\omega^{n-k}}{\omega^{n}}, ~\forall ~\alpha\in A^{1,1}(M),
\end{equation*}
where $A^{1,1}(M)$ is the space of smooth real (1,1) forms on $(M,\omega)$.
Define  a  cone $\Gamma_{2}(M)$ on   $A^{1,1}(M)$ by
\begin{equation*}
\Gamma_{2}(M)=\{ \alpha\in A^{1,1}(M)~|~ \text{$\sigma_{j}(\alpha)>0$ for $j=1,2$} \}.
\end{equation*}
Then,  (\ref{Fu-Yau equation})  is equivalent to  (\ref{fu-2}) while
the function $f(z,\vp,\de\vp)$ satisfies
\begin{align}\label{Definition of f}
f\omega^{n} &=  2\alpha\rho\wedge\omega^{n-1}+\alpha^{2}e^{-2\vp}\rho^{2}\wedge\omega^{n-2}-4n\alpha\mu\frac{\omega^{n}}{n!}\notag \\
                 & +4n\alpha^{2}e^{-\vp}\sqrt{-1}\left(\de\vp\wedge\dbar\vp\wedge\rho-\de\vp\wedge\dbar\rho
                   -\de\rho\wedge\dbar\vp+\de\dbar\rho\right)\wedge\omega^{n-2}.
\end{align}
We will use  (\ref{fu-2})  to apply the maximum principle to the quantity
\begin{equation*}
Q = \log|\de\vp|_{g}^{2}+\frac{\vp}{B},
\end{equation*}
where $B>1$ is a large uniform constant to be determined later.

Assume that $Q$ achieves a maximum at $x_{0}$.
 Let $\{e_{i}\}_{i=1}^{n}$ be a local unitary frame in a neighbourhood of $x_{0}$ such that, at $x_{0}$,
\begin{equation}\label{tilde gij}
\tilde{g}_{i\ol{j}}
= \delta_{i\ol{j}}\tilde{g}_{i\ol{i}}
= \delta_{i\ol{j}}(e^{\vp}+\alpha e^{-\vp}\rho_{i\ov{i}}+2n\alpha \vp_{i\ov{i}}).
\end{equation}
For convenience, we use the following notation:
\begin{equation*}
\hat{\omega} = e^{-\vp}\ti{\omega},
\hat{g}_{i\ov{j}} = e^{-\vp}\ti{g}_{i\ov{j}} \text{~and~}
F^{i\ov{j}} = \frac{\partial\sigma_{2}(\hat{\omega})}{\de\hat{g}_{i\ov{j}}}.
\end{equation*}
Since $\ti{g}_{i\ov{j}}(x_{0})$ is diagonal at $x_0$, it  is easy to see that
\begin{equation}\label{Expression of Fij}
F^{i\ov{j}} = \delta_{ij}F^{i\ov{i}} = \delta_{ij}e^{-\vp}\sum_{k\neq i}\ti{g}_{k\ov{k}}.
\end{equation}
By the assumption of Proposition \ref{Gradient estimate}, at the expense of increasing $C_{0}$, we have
\begin{equation}\label{Small condition}
e^{-\vp}|\de\dbar\vp|_{g}
\leq M_{0}DA_{D}
\leq \frac{1}{1000B n^{3}|\alpha|}.
\end{equation}
Combining this with (\ref{tilde gij}) and (\ref{Expression of Fij}),   we get
\begin{equation}\label{Bound of Fij}
\left|F^{i\ov{i}}-(n-1)\right| \leq \frac{1}{100}.
\end{equation}

We need to estimate the lower bound of $F^{i\ov{j}}e_{i}e_{\ov{j}}(|\de\vp|_{g}^{2})$, where we  are summing over repeated indices.
 Note
\begin{equation*}
|\de\vp|_{g}^{2} = \sum_{k}\vp_{k}\vp_{\ov{k}},
\end{equation*}
where $\vp_{k}=e_{k}(\vp)$ and $\vp_{\ov{k}}=\ov{e}_{k}(\vp)$,
in the local frame $\{e_{i}\}_{i=1}^{n}$.  Then,  at $x_{0}$,
\begin{equation*}
\begin{split}
F^{i\ov{j}}e_{i}\ov{e}_{j}(|\de\vp|_{g}^{2})
= {} & \sum_{j}F^{i\ov{i}}(|e_{i}\ov{e}_{j}(\vp)|^{2}+|e_{i}e_{j}(\vp)|^{2}) \\
     & +\sum_{k}F^{i\ov{i}}\left(e_{i}\ov{e}_{i}e_{k}(\vp)\vp_{\ov{k}}+e_{i}\ov{e}_{i}\ov{e}_{k}(\vp)\vp_{k}\right).
\end{split}
\end{equation*}
On the other hand,  by the relation  (see e.g. \cite{HL15})
\begin{equation}\label{ddbar formula}
\vp_{i\ov{j}} = \de\dbar\vp(e_{i},\ov{e}_{j}) = e_{i}\ov{e}_{j}(\vp)-[e_{i},\ov{e}_{j}]^{(0,1)}(\vp),
\end{equation}
 we have
\begin{equation*}
\begin{split}
e_{k}(\vp_{i\ov{i}}) = {} & e_{k}e_{i}\ov{e}_{i}(\vp)-e_{k}[e_{i},\ov{e}_{i}]^{(0,1)}(\vp) \\
= {} & e_{i}\ov{e}_{i}e_{k}(\vp)+[e_{k},e_{i}]\ov{e}_{i}(\vp)+e_{i}[e_{k},\ov{e}_{i}](\vp)-e_{k}[e_{i},\ov{e}_{i}]^{(0,1)}(\vp).
\end{split}
\end{equation*}
Thus combining this with (\ref{Bound of Fij}), we get
\begin{equation*}
\begin{split}
        & \sum_{k}F^{i\ov{i}}\left(e_{i}\ov{e}_{i}e_{k}(\vp)\vp_{\ov{k}}+e_{i}\ov{e}_{i}\ov{e}_{k}(\vp)\vp_{k}\right) \\[3mm]
\geq {} & \sum_{k}F^{i\ov{i}}\left(e_{k}(\vp_{i\ov{i}})\vp_{\ov{k}}+\ov{e}_{k}(\vp_{i\ov{i}})\vp_{k}\right)
          -C|\de\vp|_{g}\sum_{i,j}(|e_{i}\ov{e}_{j}(\vp)|+|e_{i}e_{j}(\vp)|)-C|\de\vp|_{g}^{2} \\
\geq {} & \sum_{k}F^{i\ov{i}}\left(e_{k}(\vp_{i\ov{i}})\vp_{\ov{k}}+\ov{e}_{k}(\vp_{i\ov{i}})\vp_{k}\right)
          -\frac{1}{10}\sum_{i,j}(|e_{i}\ov{e}_{j}(\vp)|^{2}+|e_{i}e_{j}(\vp)|^{2})-C|\de\vp|_{g}^{2}.
\end{split}
\end{equation*}
Hence, we obtain
\begin{equation}\label{Maximum principle gradient 1}
\begin{split}
F^{i\ov{j}}e_{i}\ov{e}_{j}(|\de\vp|_{g}^{2})
\geq {} & \frac{4}{5}\sum_{i,j}(|e_{i}\ov{e}_{j}(\vp)|^{2}+|e_{i}e_{j}(\vp)|^{2})-C|\de\vp|_{g}^{2} \\
        & +\sum_{k}\left(F^{i\ov{i}}e_{k}(\vp_{i\ov{i}})\vp_{\ov{k}}+\ov{e}_{k}(\vp_{i\ov{i}})\vp_{k}\right).
\end{split}
\end{equation}

Next, we use equation (\ref{fu-2})  to deal with the third order terms in (\ref{Maximum principle gradient 1}).

\begin{lemma}\label{Gradient estimate lemma}
At $x_{0}$, we have
\begin{equation}\label{gradient-lapalce}
\begin{split}
        & \sum_{k}F^{i\ov{i}}\left(e_{k}(\vp_{i\ov{i}})\vp_{\ov{k}}+\ov{e}_{k}(\vp_{i\ov{i}})\vp_{k}\right) \\[2mm]
\geq {} & -\frac{1}{5}\sum_{i,j}(|e_{i}\ov{e}_{j}(\vp)|^{2}+|e_{i}e_{j}(\vp)|^{2})
          -2(n-1){\rm Re}\left(\sum_{k}(|\de\vp|_{g}^{2})_{k}\vp_{\ov{k}}\right) \notag\\[2mm]
        & -\left(Ce^{-\vp}+\frac{1}{B}\right)|\de\vp|_{g}^{4}-C|\de\vp|_{g}^{2}-C.
\end{split}
\end{equation}
\end{lemma}

\begin{proof}
By  (\ref{fu-2}),  we have
\begin{equation}\label{fu-3}
\sigma_{2}(e^{-\vp} \tilde\omega) = e^{-2\vp}F(z,\vp,\de\vp).
\end{equation}
Differentiating    (\ref{fu-3})   along $e_{k}$ at $x_{0}$, we  get
\begin{equation*}
\begin{split}
     & F^{i\ol{j}}e_{k}(g_{i\ol{j}}+\alpha e^{-2\vp}\rho_{i\ol{j}}+2n\alpha e^{-\vp}\vp_{i\ol{j}})\\
= {} & -2\alpha n(n-1)\left(-e^{-\vp}|\de\vp|^{2}_{g}\vp_{k}+e^{-\vp}(|\de\vp|^{2}_{g})_{k}\right)+\frac{n(n-1)}{2}(e^{-2\vp}f)_{k}.
\end{split}
\end{equation*}
Then
\begin{equation*}
\begin{split}
2n\alpha F^{i\ov{i}}e_{k}(\vp_{i\ov{i}})
= {} & 2\alpha e^{-\vp}\vp_{k}F^{i\ov{i}}\rho_{i\ov{i}}-\alpha e^{-\vp}F^{i\ov{i}}e_{k}(\rho_{i\ov{i}})+2n\alpha\vp_{k}F^{i\ov{i}}\vp_{i\ov{i}} \\[1mm]
     & -2\alpha n(n-1)(|\de\vp|_{g}^{2})_{k}+2\alpha n(n-1)|\de\vp|_{g}^{2}\vp_{k} \\
     & -n(n-1)e^{-\vp}f\vp_{k}+\frac{n(n-1)}{2}e^{-\vp}f_{k}.
\end{split}
\end{equation*}
It follows that
\begin{equation}\label{Gradient estimate equation 7}
\begin{split}
        & \sum_{k}F^{i\ov{i}}\left(e_{k}(\vp_{i\ov{i}})\vp_{\ov{k}}+\ov{e}_{k}(\vp_{i\ov{i}})\vp_{k}\right) \\[1mm]
   = {} & \frac{2}{n}e^{-\vp}|\de\vp|_{g}^{2}F^{i\ov{i}}\rho_{i\ov{i}}
          -\frac{1}{n}e^{-\vp}\textrm{Re}\left(\sum_{k}F^{i\ov{i}}e_{k}(\rho_{i\ov{i}})\vp_{\ov{k}}\right)
          +2|\de\vp|_{g}^{2}F^{i\ov{i}}\vp_{i\ov{i}} \\[1mm]
        & -2(n-1)\textrm{Re}\left(\sum_{k}(|\de\vp|_{g}^{2})_{k}\vp_{\ov{k}}\right)+2(n-1)|\de\vp|_{g}^{4}
          -\frac{n-1}{\alpha}e^{-\vp}|\de\vp|_{g}^{2}f \\
        & +\frac{n-1}{2\alpha}e^{-\vp}\textrm{Re}\left(\sum_{k}f_{k}\vp_{\ov{k}}\right) \\[1mm]
\geq {} & -Ce^{-\vp}|\de\vp|_{g}^{2}-Ce^{-\vp}|\de\vp|_{g}+2|\de\vp|_{g}^{2}F^{i\ov{i}}\vp_{i\ov{i}}
          -2(n-1)\textrm{Re}\left(\sum_{k}(|\de\vp|_{g}^{2})_{k}\vp_{\ov{k}}\right) \\[1mm]
        & +2(n-1)|\de\vp|_{g}^{4}-\frac{n-1}{\alpha}e^{-\vp}|\de\vp|_{g}^{2}f
          +\frac{n-1}{2\alpha}e^{-\vp}\textrm{Re}\left(\sum_{k}f_{k}\vp_{\ov{k}}\right),
\end{split}
\end{equation}
where we used (\ref{Bound of Fij}) in the last inequality.  On the other hand, by (\ref{Definition of f}),  a direct calculation shows that
\begin{equation}\label{Gradient estimate equation 8}
\begin{split}
& -\frac{n-1}{\alpha}e^{-\vp}|\de\vp|_{g}^{2}f+\frac{n-1}{2\alpha}e^{-\vp}\textrm{Re}\left(\sum_{k}f_{k}\vp_{\ov{k}}\right) \\[1mm]
\geq {} & -C\left(e^{-2\vp}|\de\vp|^{2}_{g}+e^{-2\vp}|\de\vp|_{g}\right)\sum_{i,j}(|e_{i}\ov{e}_{j}(\vp)|+|e_{i}e_{j}(\vp)|)-Ce^{-\vp}|\de\vp|_{g}^{4} \\
        & -Ce^{-\vp}|\de\vp|_{g}^{3}-Ce^{-\vp}|\de\vp|_{g}^{2}-Ce^{-\vp}|\de\vp|_{g} \\[1mm]
\geq {} & -\frac{1}{10}\sum_{i,j}(|e_{i}\ov{e}_{j}(\vp)|^{2}+|e_{i}e_{j}(\vp)|^{2})-Ce^{-\vp}|\de\vp|_{g}^{4}-Ce^{-\vp},
\end{split}
\end{equation}
where we used the Cauchy-Schwarz inequality in the last inequality.  Thus substituting (\ref{Gradient estimate equation 8}) into (\ref{Gradient estimate equation 7}), we  derive
\begin{equation}\label{Gradient estimate equation 9}
\begin{split}
        & \sum_{k}F^{i\ov{i}}\left(e_{k}(\vp_{i\ov{i}})\vp_{\ov{k}}+\ov{e}_{k}(\vp_{i\ov{i}})\vp_{k}\right) \\[2mm]
\geq {} & 2|\de\vp|_{g}^{2}F^{i\ov{i}}\vp_{i\ov{i}}+2(n-1)|\de\vp|_{g}^{4}
          -2(n-1)\textrm{Re}\left(\sum_{k}(|\de\vp|_{g}^{2})_{k}\vp_{\ov{k}}\right) \\[2mm]
        & -\frac{1}{10}\sum_{i,j}(|e_{i}\ov{e}_{j}(\vp)|^{2}+|e_{i}e_{j}(\vp)|^{2})-Ce^{-\vp}|\de\vp|_{g}^{4}-Ce^{-\vp}.
\end{split}
\end{equation}

By (\ref{Expression of Fij}), we  have
\begin{equation}\label{Gradient estimate equation 1}
\begin{split}
        & 2|\de\vp|_{g}^{2}F^{i\ol{i}}\vp_{i\ol{i}} \\[2mm]
   &= 2|\de\vp|^{2}_{g}\sum_{i}\sum_{k\neq i}(g_{k\ol{k}}+\alpha e^{-2\vp}\rho_{k\ol{k}}+2n\alpha e^{-\vp}\vp_{k\ol{k}})\vp_{i\ol{i}} \\
  & = 2(n-1)|\de\vp|_{g}^{2}\Delta\vp+4n\alpha|\de\vp|_{g}^{2}e^{-\vp}\sum_{i\neq k}\vp_{i\ov{i}}\vp_{k\ov{k}}
          +2\alpha e^{-2\vp}|\de\vp|_{g}^{2}\sum_{i\neq k}\vp_{i\ov{i}}\rho_{k\ov{k}} \\
&\geq  2(n-1)|\de\vp|_{g}^{2}\Delta\vp-4n^{2}(n-1)|\alpha|e^{-\vp}|\de\vp|_{g}^{2}|\de\dbar\vp|_{g}^{2}
          -Ce^{-2\vp}|\de\vp|^{2}_{g}|\de\dbar\vp|_{g}.
\end{split}
\end{equation}
Note that by (\ref{Fu-Yau equation}) it holds
\begin{equation*}
\frac{\ddbar(e^{\vp}\omega-\alpha e^{-\vp}\rho)\wedge\omega^{n-2}}{\omega^{n}}
\geq -n|\alpha||\de\dbar\vp|_{g}^{2}-C,
\end{equation*}
which implies
\begin{equation}\label{Gradient estimate equation 2}
\Delta\vp+|\de\vp|_{g}^{2}
\geq -n|\alpha|e^{-\vp}|\de\dbar\vp|_{g}^{2}-Ce^{-2\vp}|\de\dbar\vp|_{g}-Ce^{-2\vp}|\de\vp|_{g}^{2}-C.
\end{equation}
Then substituting (\ref{Gradient estimate equation 2}) into (\ref{Gradient estimate equation 1}),  we get
\begin{equation*}
\begin{split}
2|\de\vp|_{g}^{2}F^{i\ov{i}}\vp_{i\ov{i}}
\geq {} & -2(n-1)|\de\vp|_{g}^{4}-5n^{3}|\alpha|e^{-\vp}|\de\vp|_{g}^{2}|\de\dbar\vp|_{g}^{2} \\
     {} & -Ce^{-2\vp}|\de\vp|_{g}^{2}|\de\dbar\vp|_{g}-Ce^{-2\vp}|\de\vp|_{g}^{4}-C|\de\vp|_{g}^{2}.
\end{split}
\end{equation*}
Thus  by  (\ref{Small condition}) and the Cauchy-Schwarz inequality, we derive
\begin{equation}\label{Gradient estimate equation 10}
\begin{split}
& 2|\de\vp|_{g}^{2}F^{i\ov{i}}\vp_{i\ov{i}}+2(n-1)|\de\vp|_{g}^{4} \\[2mm]
\geq {} & -5n^{3}|\alpha|(e^{-\vp}|\de\dbar\vp|_{g})(|\de\vp|_{g}^{4}+|\de\dbar\vp|_{g}^{2}) \\[2mm]
        & -C(e^{-\vp}|\de\dbar\vp|_{g})(e^{-\vp}|\de\vp|_{g}^{2})-Ce^{-2\vp}|\de\vp|_{g}^{4}-C|\de\vp|_{g}^{2} \\[2mm]
\geq {} & -\frac{1}{B}|\de\dbar\vp|_{g}^{2}-\left(Ce^{-2\vp}+\frac{1}{B}\right)|\de\vp|_{g}^{4}-C|\de\vp|_{g}^{2} \\
\geq {} & -\frac{1}{10}\sum_{i,j}|e_{i}\ov{e}_{j}(\vp)|^{2}-\left(Ce^{-2\vp}+\frac{1}{B}\right)|\de\vp|_{g}^{4}-C|\de\vp|_{g}^{2}.
\end{split}
\end{equation}
Combining  (\ref{Gradient estimate equation 9}) and (\ref{Gradient estimate equation 10}), we prove Lemma \ref{Gradient estimate lemma}  immediately.
\end{proof}

By  (\ref{Maximum principle gradient 1}) and Lemma  \ref{Gradient estimate lemma},   we  get a  lower bound for $F^{i\ov{j}}e_{i}\ov{e}_{j}(|\de\vp|_{g}^{2})$ at $x_{0}$ as follows,
\begin{equation}\label{Maximum principle gradient 2}
\begin{split}
&F^{i\ov{j}}e_{i}\ov{e}_{j}(|\de\vp|_{g}^{2})\\
&\geq {}  \frac{3}{5}\sum_{i,j}(|e_{i}\ov{e}_{j}(\vp)|^{2}+|e_{i}e_{j}(\vp)|^{2})
          -2(n-1)\textrm{Re}\left(\sum_{k}(|\de\vp|_{g}^{2})_{k}\vp_{\ov{k}}\right) \\[2mm]
        & -\left(Ce^{-\vp}+\frac{1}{B}\right)|\de\vp|_{g}^{4}-C|\de\vp|_{g}^{2}-C.
\end{split}
\end{equation}
Now we are in a position to prove Proposition \ref{Gradient estimate}.

\begin{proof}[Proof of Proposition \ref{Gradient estimate}]
Without loss of generality, we assume that $|\de\vp|_{g}^{2}\geq1$.
By (\ref{Maximum principle gradient 2}) and the maximum principle, at $x_{0}$,  we see that
\begin{equation}\label{Gradient estimate equation 3}
\begin{split}
0 \geq {} & F^{i\ov{j}}e_{i}\ov{e}_{j}(Q) \\
    =  {} & \frac{F^{i\ov{j}}e_{i}\ov{e}_{j}(|\de\vp|_{g}^{2})}{|\de\vp|_{g}^{2}}
            -\frac{F^{i\ov{j}}e_{i}(|\de\vp|_{g}^{2})\ov{e}_{j}(|\de\vp|_{g}^{2})}{|\de\vp|_{g}^{4}}
            +\frac{1}{B}F^{i\ov{j}}e_{i}\ov{e}_{j}(\vp) \\
  \geq {} & \frac{1}{2|\de\vp|_{g}^{2}}\sum_{i,j}(|e_{i}\ov{e}_{j}(\vp)|^{2}+|e_{i}e_{j}(\vp)|^{2})
            -\frac{2(n-1)\textrm{Re}\left(\sum_{k}(|\de\vp|_{g}^{2})_{k}\vp_{\ov{k}}\right)}{|\de\vp|_{g}^{2}} \\
          & -\frac{F^{i\ov{i}}|e_{i}(|\de\vp|_{g}^{2})|^{2}}{|\de\vp|_{g}^{4}}
            -\left(Ce^{-\vp}+\frac{1}{B}\right)|\de\vp|_{g}^{2}-C+\frac{1}{B}F^{i\ov{i}}e_{i}\ov{e}_{i}(\vp).
\end{split}
\end{equation}
The second and third terms in (\ref{Gradient estimate equation 3}) can be controlled  by   the relation $dQ(x_{0})=0$.  Namely, we have
\begin{equation}\label{Gradient estimate equation 4}
-\frac{2(n-1)\textrm{Re}\left(\sum_{k}(|\de\vp|_{g}^{2})_{k}\vp_{\ov{k}}\right)}{|\de\vp|_{g}^{2}}
=\frac{2(n-1)}{B}|\de\vp|_{g}^{2}
\end{equation}
and
\begin{equation}\label{Gradient estimate equation 5}
-\frac{F^{i\ov{i}}|e_{i}(|\de\vp|_{g}^{2})|^{2}}{|\de\vp|_{g}^{4}}
= -\frac{1}{B^{2}}F^{i\ov{i}}\vp_{i}\vp_{\ov{i}}
\geq -\frac{C}{B^{2}}|\de\vp|_{g}^{2},
\end{equation}
where we used (\ref{Bound of Fij}) in the last inequality.  On the other hand,  by  (\ref{Bound of Fij}) and the Cauchy-Schwarz inequality, we have
\begin{equation}\label{Gradient estimate equation 6}
\frac{1}{B}F^{i\ov{i}}e_{i}\ov{e}_{i}(\vp)
\geq -\frac{1}{4|\de\vp|_{g}^{2}}\sum_{i,j}|e_{i}\ov{e}_{j}(\vp)|^{2}-\frac{C}{B^{2}}|\de\vp|_{g}^{2}.
\end{equation}
Thus substituting (\ref{Gradient estimate equation 4}), (\ref{Gradient estimate equation 5}), (\ref{Gradient estimate equation 6}) into (\ref{Gradient estimate equation 3}), we get
\begin{equation}\label{dradient-lalpace-2}
\begin{split}
0 \geq {} & \frac{1}{4|\de\vp|_{g}^{2}}\sum_{i,j}\left(|e_{i}\ov{e}_{j}(\vp)|^{2}+|e_{i}e_{j}(\vp)|^{2}\right)-C_{0} \\
          & +\left(\frac{2n-3}{B}-\frac{C_{0}}{B^{2}}-C_{0}e^{-\vp}\right)|\de\vp|_{g}^{2},
\end{split}
\end{equation}
where $C_{0}$ is a uniform constant.

We choose the number $B = 2C_{0}$ in (\ref{dradient-lalpace-2}). Moreover, by the assumption  in the proposition we  may also assume
\begin{equation*}
C_{0}e^{-\vp} \leq \frac{1}{8C_{0}}.
\end{equation*}
Then,  we get
\begin{equation*}
|\de\vp|_{g}^{2}(x_{0}) \leq 8C_{0}^{2}.
\end{equation*}
Hence, by Proposition \ref{Zero order estimate},  we obtain
\begin{equation*}
\max_{M}|\de\vp|_{g}^{2} \leq e^{\frac{1}{B}(\sup_{M}\vp-\inf_{M}\vp)}|\de\vp|_{g}^{2}(x_{0}) \leq C,
\end{equation*}
as desired.
\end{proof}

The following  lemma  will  be used  in the next section.
\begin{lemma}\label{Maximum principle gradient lemma}
For a uniform constant $C_{1}$, we have
\begin{equation*}
F^{i\ov{j}}e_{i}\ov{e}_{j}(|\de\vp|_{g}^{2})
\geq \frac{1}{2}\sum_{i,j}(|e_{i}\ov{e}_{j}(\vp)|^{2}+|e_{i}e_{j}(\vp)|^{2})-C_{1}.
\end{equation*}
\end{lemma}

\begin{proof}
This lemma is an immediate consequence of (\ref{Maximum principle gradient 2}), Proposition \ref{Gradient estimate} and the Cauchy-Schwarz inequality.
\end{proof}

\section{Second order estimate}

This section is denoted to  $C^2$-estimate. We prove

\begin{proposition}\label{Second order estimate}
Let $\vp$ be a solution of (\ref{Fu-Yau equation}) satisfying (\ref{Elliptic condition}) and $\frac{A}{M_0}\leq e^{-\vp}\leq M_0A$ for some uniform constant $M_0$. There exist uniform constants $D_{0}$ and $C_{0}$ such that if
\begin{equation*}
|\de\dbar\vp|_{g} \leq D, ~~D_{0}\leq D \text{~and~} A\leq A_{D}:=\frac{1}{C_{0}M_{0}D},
\end{equation*}
then
\begin{equation*}
|\de\dbar\vp|_{g} \leq \frac{D}{2}.
\end{equation*}
\end{proposition}

We consider the following quantity
\begin{equation*}
Q = |\de\dbar\vp|_{g}^{2} + B|\de\vp|_{g}^{2},
\end{equation*}
where $B>1$ is a uniform constant to be determined later.  As in Section 3,  we assume that $Q(x_{0})=\max_{M}Q$ and a local $g$-unitary frame $\{e_{i}\}_{i=1}^{n}$ for $T_{\mathbb{C}}^{(1,0)}M$  around $x_{0}$  such that $\ti{g}_{i\ov{j}}(x_{0})$ is diagonal.  By  the following notations,
\begin{equation*}
\hat{\omega} = e^{-\vp}\ti{\omega},
\hat{g}_{i\ov{j}} = e^{-\vp}\ti{g}_{i\ov{j}},
F^{i\ov{j}} = \frac{\de{\sigma_{2}(\hat{\omega})}}{\de\hat{g}_{i\ov{j}}}
\text{~and~}
F^{i\ov{j},k\ov{l}} = \frac{\de^{2}{\sigma_{2}(\hat{\omega})}}{\de\hat{g}_{i\ov{j}}\de\hat{g}_{k\ov{l}}},
\end{equation*}
we have
\begin{equation*}
F^{i\ov{j}} = \delta_{ij}F^{i\ov{i}} = \delta_{ij}e^{-\vp}\sum_{k\neq i}\ti{g}_{k\ov{k}}
\end{equation*}
and
\begin{equation*}
F^{i\ov{j},k\ov{l}}=
\left\{\begin{array}{ll}
 1, & \text{if $i=j$, $k=l$, $i\neq k$;}\\[1mm]
-1, & \text{if $i=l$, $k=j$, $i\neq k$;}\\[1mm]
 0, & \text{\quad\quad~otherwise.}
\end{array}\right.
\end{equation*}
By the assumption of Proposition \ref{Second order estimate}, at the expense of increasing $C_{0}$, we may also assume that
\begin{equation}\label{Second order estimate equation 6}
e^{-\vp}|\de\dbar\vp|_{g}\leq \frac{1}{1000n^{3}|\alpha|B}.
\end{equation}
Hence, we get
\begin{equation}\label{Second order estimate equation 1}
|F^{i\ov{i}}-(n-1)| \leq \frac{1}{100}
\text{~and~} |F^{i\ov{j},k\ov{l}}|\leq 1.
\end{equation}

We need the following lemma.

\begin{lemma}\label{Second order estimate lemma}
At $x_{0}$, we have
\begin{equation*}
\begin{split}
|F^{i\ov{i}}e_{i}\ov{e}_{i}(\vp_{k\ov{l}})|
\leq {} & 8n|\alpha| e^{-\vp}\sum_{i,j,p}|e_{p}e_{i}\ov{e}_{j}(\vp)|^{2}+C\sum_{i,j,p}|e_{p}e_{i}\ov{e}_{j}(\vp)| \\
        & +C\sum_{i,j}(|e_{i}\ov{e}_{j}(\vp)|^{2}+|e_{i}e_{j}(\vp)|^{2})+C.
\end{split}
\end{equation*}
\end{lemma}

\begin{proof}
Differentiating (\ref{fu-3}) twice  along $e_{k}$ and  $\bar {e}_{l}$ at $x_{0}$, we have
\begin{equation}\label{second-derivative-sigma2}
\begin{split}
&F^{i\ov{j},p\ov{q}}e_{k}(e^{-\vp}\ti{g}_{i\ov{j}})\ov{e}_{l}  (e^{-\vp}\ti{g}_{p\ov{q}})
+F^{i\ov{j}}e_{k}\ov{e}_{l}(e^{-\vp}\ti{g}_{i\ov{j}}) \\
&= -2n(n-1)\alpha e_{k}\ov{e}_{l}(e^{-\vp}|\de\vp|_{g}^{2})+\frac{n(n-1)}{2}e_{k}\ov{e}_{l}(e^{-2\vp}f).
\end{split}
\end{equation}
Let
\begin{equation*}
\begin{split}
I_{1} = {} & -F^{i\ov{j},p\ov{q}}e_{k}(e^{-\vp}\ti{g}_{i\ov{j}})\ov{e}_{l}(e^{-\vp}\ti{g}_{p\ov{q}}), \\[1.5mm]
I_{2} = {} & -2n(n-1)\alpha e_{k}\ov{e}_{l}(e^{-\vp}|\de\vp|_{g}^{2}), \\
I_{3} = {} & \frac{n(n-1)}{2}e_{k}\ov{e}_{l}(e^{-2\vp}f).
\end{split}
\end{equation*}
Then (\ref{second-derivative-sigma2}) becomes
\begin{equation}\label{Second order estimate equation 2}
F^{i\ov{j}}e_{k}\ov{e}_{l}(e^{-\vp}\ti{g}_{i\ov{j}})
= I_{1}+I_{2}+I_{3}.
\end{equation}
We estimate each term in (\ref{Second order estimate equation 2}) below.   For $I_{1}$, by (\ref{Second order estimate equation 1}), Proposition \ref{Gradient estimate} and the Cauchy-Schwarz inequality, we have
\begin{equation*}
\begin{split}
|I_{1}| \leq {} & \sum_{i,j,k}\left|e_{k}(\alpha e^{-2\vp}\rho_{i\ov{j}}+2n\alpha e^{-\vp}\vp_{i\ov{j}})\right|^{2} \\
        \leq {} & 2\sum_{i,j,k}\left|e_{k}(2n\alpha e^{-\vp}\vp_{i\ov{j}})\right|^{2}
                  +2\sum_{i,j,k}\left|e_{k}(\alpha e^{-2\vp}\rho_{i\ov{j}})\right|^{2} \\
        \leq {} & 8n^{2}\alpha^{2}e^{-2\vp}\sum_{i,j,k}\left|e_{k}e_{i}\ov{e}_{j}(\vp)-e_{k}[e_{i},\ov{e}_{j}]^{(0,1)}(\vp)
                  -\vp_{k}\vp_{i\ov{j}}\right|^{2}+Ce^{-4\vp} \\
        \leq {} & 16n^{2}\alpha^{2}e^{-2\vp}\sum_{i,j,k}|e_{k}e_{i}\ov{e}_{j}(\vp)|^{2}
                  +Ce^{-2\vp}\sum_{i,j}\left(|e_{i}\ov{e}_{j}(\vp)|^{2}+|e_{i}{e}_{j}(\vp)|^{2}\right)+Ce^{-2\vp},
\end{split}
\end{equation*}
where we used (\ref{ddbar formula}) in the last inequality.  Similarly,  for $I_{2}$ and $I_{3}$, we get
\begin{equation*}
|I_{2}| \leq Ce^{-\vp}\sum_{i,j,p}|e_{p}e_{i}\ov{e}_{j}(\vp)|
             +Ce^{-\vp}\sum_{i,j}\left(|e_{i}\ov{e}_{j}(\vp)|^{2}+|e_{i}{e}_{j}(\vp)|^{2}\right)+Ce^{-\vp}
\end{equation*}
and
\begin{equation*}
\begin{split}
|I_{3}|    = {} & \frac{n(n-1)}{2}e^{-2\vp}\left|4\vp_{k}\vp_{\ov{l}}f-2e_{k}\ov{e}_{l}(\vp)f
                  -2\vp_{\ov{l}}f_{k}-2\vp_{k}f_{\ov{l}}+e_{k}\ov{e}_{l}f\right| \\
        \leq {} & Ce^{-2\vp}\sum_{i,j,p}|e_{p}e_{i}\ov{e}_{j}(\vp)|+Ce^{-2\vp}\sum_{i,j}\left(|e_{i}\ov{e}_{j}(\vp)|^{2}
                  +|e_{i}{e}_{j}(\vp)|^{2}\right)+Ce^{-2\vp},
\end{split}
\end{equation*}
where we   used Proposition \ref{Gradient estimate} and (\ref{Definition of f}).  Thus substituting these estimates into (\ref{Second order estimate equation 2}),   we obtain
\begin{equation}\label{Second order estimate equation 7}
\begin{split}
&|F^{i\ov{i}}e_{k}\ov{e}_{l}(e^{-\vp}\ti{g}_{i\ov{i}})|\\
&\leq  16n^{2}\alpha^{2}e^{-2\vp}\sum_{i,j,p}|e_{p}e_{i}\ov{e}_{j}(\vp)|^{2}+Ce^{-\vp}\sum_{i,j,p}|e_{p}e_{i}\ov{e}_{j}(\vp)| \\
        & +Ce^{-\vp}\sum_{i,j}(|e_{i}\ov{e}_{j}(\vp)|^{2}+|e_{i}e_{j}(\vp)|^{2})+Ce^{-\vp}.
\end{split}
\end{equation}

On the other hand,  by the definition of $\ti{g}_{i\ov{i}}$ and (\ref{ddbar formula}),  we have
\begin{equation*}
\begin{split}
F^{i\ov{i}}e_{k}\ov{e}_{l}(e^{-\vp}\ti{g}_{i\ov{i}})
= {} & \alpha F^{i\ov{i}}e_{k}\ov{e}_{l}(e^{-2\vp}\rho_{i\ov{i}})+2n\alpha F^{i\ov{i}}e_{k}\ov{e}_{l}(e^{-\vp}\vp_{i\ov{i}}) \\
= {} & \alpha F^{i\ov{i}}e_{k}\ov{e}_{l}(e^{-2\vp}\rho_{i\ov{i}})+2n\alpha F^{i\ov{i}}e_{k}\ov{e}_{l}(e^{-\vp}e_{i}\ov{e}_{i}(\vp)) \\
     & -2n\alpha F^{i\ov{i}}e_{k}\ov{e}_{l}(e^{-\vp}[e_{i},\ov{e}_{i}]^{(0,1)}(\vp)).
\end{split}
\end{equation*}
Then by  (\ref{Second order estimate equation 1}) and Proposition \ref{Gradient estimate},  it  follows that
\begin{equation}
\begin{split}
|2n\alpha e^{-\vp}F^{i\ov{i}}e_{k}\ov{e}_{l}e_{i}\ov{e}_{i}(\vp)|
\leq {} & |F^{i\ov{i}}e_{k}\ov{e}_{l}(e^{-\vp}\ti{g}_{i\ov{i}})|+Ce^{-\vp}\sum_{i,j,p}|e_{p}e_{i}\ov{e}_{j}(\vp)| \notag\\
        & +Ce^{-\vp}\sum_{i,j}(|e_{i}\ov{e}_{j}(\vp)|^{2}+|e_{i}e_{j}(\vp)|^{2})+Ce^{-\vp}.
\end{split}
\end{equation}
Thus substituting (\ref{Second order estimate equation 7}) into the above inequality, we  derive
\begin{equation}\label{Second order estimate equation 5}
\begin{split}
|F^{i\ov{i}}e_{k}\ov{e}_{l}e_{i}\ov{e}_{i}(\vp)|
\leq {} & 8n|\alpha|e^{-\vp}\sum_{i,j,p}|e_{p}e_{i}\ov{e}_{j}(\vp)|^{2}+C\sum_{i,j,p}|e_{p}e_{i}\ov{e}_{j}(\vp)| \\
        & +C\sum_{i,j}(|e_{i}\ov{e}_{j}(\vp)|^{2}+|e_{i}e_{j}(\vp)|^{2})+C.
\end{split}
\end{equation}

Note  that
\begin{equation}\label{Second order estimate equation 4}
\begin{split}
       e_{i}\ov{e}_{i}e_{k}\ov{e}_{l}(\vp)
= {} & e_{k}\ov{e}_{l}e_{i}\ov{e}_{i}(\vp)+e_{k}[e_{i},\ov{e}_{l}]\ov{e}_{i}(\vp)+[e_{i},e_{k}]\ov{e}_{l}\ov{e}_{i}(\vp) \\
     & +e_{i}e_{k}[\ov{e}_{i},\ov{e}_{l}](\vp)+e_{i}[\ov{e}_{i},e_{k}]\ov{e}_{l}(\vp).
\end{split}
\end{equation}
Since $(M,\omega)$ is Hermitian, near $x_{0}$, $[e_{i},e_{k}]$ is a $(1,0)$ vector field and $[\ov{e}_{i},\ov{e}_{l}]$ is a $(0,1)$ vector field.    By  (\ref{Second order estimate equation 1}) and (\ref{Second order estimate equation 5}),  we see that
\begin{equation*}
\begin{split}
|F^{i\ov{i}}e_{i}\ov{e}_{i}e_{k}\ov{e}_{l}(\vp)|
\leq {} & |F^{i\ov{i}}e_{k}\ov{e}_{l}e_{i}\ov{e}_{i}(\vp)|+C\sum_{i,j,p}|e_{p}e_{i}\ov{e}_{j}(\vp)| \\
        & +C\sum_{i,j}(|e_{i}\ov{e}_{j}(\vp)|+|e_{i}e_{j}(\vp)|) \\
\leq {} & 8n|\alpha|e^{-\vp}\sum_{i,j,p}|e_{p}e_{i}\ov{e}_{j}(\vp)|^{2}+C\sum_{i,j,p}|e_{p}e_{i}\ov{e}_{j}(\vp)| \\
        & +C\sum_{i,j}(|e_{i}\ov{e}_{j}(\vp)|^{2}+|e_{i}e_{j}(\vp)|^{2})+C.
\end{split}
\end{equation*}
As a consequence, we obtain
\begin{equation*}
\begin{split}
|F^{i\ov{i}}e_{i}\ov{e}_{i}(\vp_{k\ov{l}})|
\leq {} & |F^{i\ov{i}}e_{i}\ov{e}_{i}e_{k}\ov{e}_{l}(\vp)|+|F^{i\ov{i}}e_{i}\ov{e}_{i}[e_{k},\ov{e}_{l}]^{(0,1)}(\vp)| \\[2.5mm]
\leq {} & 8n|\alpha|e^{-\vp}\sum_{i,j,p}|e_{p}e_{i}\ov{e}_{j}(\vp)|^{2}+C\sum_{i,j,p}|e_{p}e_{i}\ov{e}_{j}(\vp)| \\
        & +C\sum_{i,j}(|e_{i}\ov{e}_{j}(\vp)|^{2}+|e_{i}e_{j}(\vp)|^{2})+C.
\end{split}
\end{equation*}
The lemma is proved.
\end{proof}

\begin{proof}[Proof of  Proposition \ref{Second order estimate}]  By Lemma  \ref{Second order estimate lemma} and the Cauchy-Schwarz inequality, at $x_{0}$, we have
\begin{equation*}
\begin{split}
F^{i\ov{i}}e_{i}\ov{e}_{i}(|\de\dbar\vp|_{g}^{2})
   = {} & 2\sum_{k,l}F^{i\ov{i}}e_{i}\ov{e}_{i}(\vp_{k\ov{l}})\vp_{l\ov{k}}+2\sum_{k,l}F^{i\ov{i}}e_{i}(\vp_{k\ov{l}})\ov{e}_{i}(\vp_{l\ov{k}}) \\
\geq {} & -2|\de\dbar\vp|_{g}\sum_{k,l}|F^{i\ov{i}}e_{i}\ov{e}_{i}(\vp_{k\ov{l}})|+\frac{1}{2}\sum_{i,j,p}|e_{p}e_{i}\ov{e}_{j}(\vp)|^{2} \\
        & -C\sum_{i,j}(|e_{i}\ov{e}_{j}(\vp)|^{2}+|e_{i}e_{j}(\vp)|^{2})-C \\
\geq {} & \left(\frac{1}{4}-8n^{3}|\alpha|e^{-\vp}|\de\dbar\vp|_{g}\right)\sum_{i,j,p}|e_{p}e_{i}\ov{e}_{j}(\vp)|^{2}-C \\
        & -C(|\de\dbar\vp|_{g}+1)\sum_{i,j}(|e_{i}\ov{e}_{j}(\vp)|^{2}+|e_{i}{e}_{j}(\vp)|^{2}).
\end{split}
\end{equation*}
Recalling (\ref{Second order estimate equation 6}) and $|\de\dbar\vp|_{g}\leq D$.  Thus
\begin{equation*}
\begin{split}
F^{i\ov{i}}e_{i}\ov{e}_{i}(|\de\dbar\vp|_{g}^{2})
\geq -C_{0}(D+1)\sum_{i,j}(|e_{i}\ov{e}_{j}(\vp)|^{2}+|e_{i}{e}_{j}(\vp)|^{2})-C_{0},
\end{split}
\end{equation*}
where $C_{0}$ is a uniform constant. On the other hand, by Lemma \ref{Maximum principle gradient lemma}, we have
\begin{equation*}
F^{i\ov{i}}e_{i}\ov{e}_{i}(|\de\vp|_{g}^{2})
\geq \frac{1}{2}\sum_{i,j}(|e_{i}\ov{e}_{j}(\vp)|^{2}+|e_{i}{e}_{j}(\vp)|^{2})-C_{1}.
\end{equation*}
Hence, by the maximum principle, at $x_{0}$, we get
\begin{equation*}
\begin{split}
0 \geq {} & F^{i\ov{i}}e_{i}\ov{e}_{i}(Q) \\[1mm]
     = {} & F^{i\ov{i}}e_{i}\ov{e}_{i}(|\de\dbar\vp|_{g}^{2})+BF^{i\ov{i}}e_{i}\ov{e}_{i}(|\de\vp|_{g}^{2}) \\
  \geq {} & \left(\frac{B}{2}-C_{0}D-C_{0}\right)\sum_{i,j}(|e_{i}\ov{e}_{j}(\vp)|^{2}+|e_{i}{e}_{j}(\vp)|^{2})-C_{0}-C_{1}B.
\end{split}
\end{equation*}
Choose  $B=8C_{0}D+8C_{0}$.   It follows that
\begin{equation*}
|\de\dbar\vp|_{g}^{2}(x_{0}) \leq C.
\end{equation*}
Therefore,    by  Proposition \ref{Gradient estimate}, at the expense of increasing $D_{0}$, we obtain
\begin{equation*}
\max_{M}|\de\dbar\vp|_{g}^{2} \leq |\de\dbar\vp|_{g}^{2}(x_{0})+BC \leq CD \leq \frac{D^{2}}{4}.
\end{equation*}

\end{proof}

\section{Proofs of Theorem \ref{Existence Theorem} and Theorem \ref{Uniqueness Theorem}}

In this section, we prove Theorem \ref{Existence Theorem} and Theorem \ref{Uniqueness Theorem}. We use the continuity method and consider the family of equations ($t\in[0,1]$),
\begin{equation}\label{Fu-Yau equation t}
\begin{split}
\ddbar(e^{\vp}\omega & -t\alpha e^{-\vp}\rho)\wedge\omega^{n-2} \\
& +n\alpha\ddbar\vp\wedge\ddbar\vp\wedge\omega^{n-2}+t\mu\frac{\omega^{n}}{n!}=0,
\end{split}
\end{equation}
where $\vp$ satisfies the elliptic condition,
\begin{equation}\label{Elliptic condition t}
e^{\vp}\omega+t\alpha e^{-\vp}\rho+2n\alpha\ddbar\vp \in \Gamma_{2}(M)
\end{equation}
and the normalization condition
\begin{equation}\label{Normalization condition t}
\|e^{-\vp}\|_{L^{1}} = A.
\end{equation}
We shall prove that  (\ref{Fu-Yau equation t}) is solvable for any $t\in[0,1]$. As the Fu-Yau equation (\ref{Fu-Yau equation}),
 (\ref{Fu-Yau equation t}) is equivalent  to a $2$-nd Hessian type equation as (\ref{fu-2}).

For a fixed $\beta\in (0,1)$, we define the following sets of functions on $M$,
\begin{equation*}
\begin{split}
B & = \{ \vp\in C^{2,\beta}(M) ~|~ \|e^{-\vp}\|_{L^{1}}=A \},\\[2mm]
B_{1} & = \{ (\vp,t)\in B\times[0,1] ~|~ \text{$\vp$ satisfies (\ref{Elliptic condition t})} \},\\
B_{2} & = \{ u\in C^{\beta}(M) ~|~ \int_{M}u\omega^{n}=0 \}.
\end{split}
\end{equation*}
Then $B_{1}$ is an open subset of $B\times[0,1]$. Since $\int_{M}\mu\omega^{n}=0$, we introduce a map $\Phi:B_{1}\rightarrow B_{2}$,
\begin{equation*}
\begin{split}
\Phi(\vp,t)\omega^{n} = {} & \ddbar(e^{\vp}\omega-t\alpha e^{-\vp}\rho)\wedge\omega^{n-2} \\
& +n\alpha\ddbar\vp\wedge\ddbar\vp\wedge\omega^{n-2}+t\mu\frac{\omega^{n}}{n!}.
\end{split}
\end{equation*}
Let $I$ be the set
\begin{equation*}
\{ t\in [0,1] ~|~ \text{there exists $(\vp,t)\in B_{1}$ such that $\Phi(\vp,t)=0$} \}.
\end{equation*}
Thus,  to prove Theorem \ref{Existence Theorem}, it suffices to prove that $I=[0,1]$.   Note that $\vp_{0}=-\ln A$ is a solution of  (\ref{Fu-Yau equation t})  at  $t=0$.  Hence, we have $0\in I$.  In the following,   we  prove that the set $I$ is both  open and closed.

\subsection{Openness}
Suppose that $\hat{t}\in I$.  By the definition of the set $I$, there exists $(\hat{\vp},\hat{t})\in B_{1}$ such that $\Phi(\hat{\vp},\hat{t})=0$. Let $(D_{\vp}\Phi)_{(\hat{\vp},\hat{t})}$ be the linearized operator of $\Phi$ at $\hat{\vp}$. Then we have
\begin{equation*}
(D_{\vp}\Phi)_{(\hat{\vp},\hat{t})}: \{ u\in C^{2,\beta}(M) ~|~ \int_{M}ue^{-\hat{\vp}}\omega^{n}=0 \}
\rightarrow \{ v\in C^{\beta}(M) ~|~ \int_{M}v\omega^{n}=0 \}
\end{equation*}
and
\begin{equation*}
\begin{split}
(D_{\vp}\Phi)_{(\hat{\vp},\hat{t})}(u)\omega^{n} = \ddbar(ue^{\hat{\vp}}\omega & +\hat{t}\alpha ue^{-\hat{\vp}}\rho)\wedge\omega^{n-2} \\
& +2n\alpha\ddbar\hat{\vp}\wedge\ddbar u\wedge\omega^{n-2}.
\end{split}
\end{equation*}
We use the implicit function theorem to  prove the openness of $I$.  It suffices to prove that $(D_{\vp}\Phi)_{(\hat{\vp},\hat{t})}$ is injective and surjective.  For convenience, we let   $L:C^{2,\beta}(M)\rightarrow C^{\beta}(M)$ be an extension  operator of  $(D_{\vp}\Phi)_{(\hat{\vp},\hat{t})}$.  First we compute the formal $L^{2}$-adjoint of $L$ in the following.

  For any $u,v\in C^{\infty}(M)$, we have
\begin{equation*}
\begin{split}
     & \int_{M}vL(u)\omega^{n} \\
= {} & \int_{M}v\left(\ddbar(ue^{\hat{\vp}}\omega+\hat{t}\alpha ue^{-\hat{\vp}}\rho)
       +2n\alpha\ddbar\hat{\vp}\wedge\ddbar u\right)\wedge\omega^{n-2} \\
= {} & \int_{M}u\left((e^{\hat{\vp}}\omega+\hat{t}\alpha e^{-\hat{\vp}}\rho)\wedge\ddbar v
       +2n\alpha\ddbar\hat{\vp}\wedge\ddbar v\right)\wedge\omega^{n-2}.
\end{split}
\end{equation*}
This  implies that
\begin{equation*}
L^{*}(v)\omega^{n} = \ddbar v\wedge\left((e^{\hat{\vp}}\omega+\hat{t}\alpha e^{-\hat{\vp}}\rho)+2n\alpha\ddbar\hat{\vp}\right)\wedge\omega^{n-2}.
\end{equation*}
By the strong maximum principle, it follows
\begin{equation*}
\textrm{Ker}L^{*} = \{ \text{Constant functions on $M$} \}.
\end{equation*}
Since the index of $L$ is zero,  we see  that $\dim\textrm{Ker}L=1$. Combining this with the theory of linear elliptic equations, there exists a positive function $u_{0}\in C^{2,\beta}(M)$ such that
\begin{equation*}
\textrm{Ker}L = \{ cu_{0} ~|~ c\in\mathbf{R} \}.
\end{equation*}
Hence,
\begin{equation*}
\int_{M}u_{0}e^{-\hat{\vp}}\omega^{n}>0 \text{~and~}
u_{0} \notin \{ u\in C^{2,\beta}(M) ~|~ \int_{M}ue^{-\hat{\vp}}\omega^{n}=0 \},
\end{equation*}
which implies $(D_{\vp}\Phi)_{(\hat{\vp},\hat{t})}$ is injective.

Next, for any $v\in C^{\beta}(M)$ such that $\int_{M}v\omega^{n}=0$, by the Fredholm alternative, there exists a weak solution $u$ of the equation $Lu=v$.   Moreover, by  the theory of linear elliptic equations, we see  that $u\in C^{2,\beta}(M)$.
Taking
\begin{equation*}
c_{0} = -\frac{\int_{M}ue^{-\hat{\vp}}\omega^{n}}{\int_{M}u_{0}e^{-\hat{\vp}}\omega^{n}}.
\end{equation*}
Then
\begin{equation*}
(D_{\vp}\Phi)_{(\hat{\vp},\hat{t})}(u+c_{0}u_{0})=L(u+c_{0}u_{0})=v \text{~and~} \int_{M}(u+c_{0}u_{0})e^{-\hat{\vp}}\omega^{n}=0,
\end{equation*}
which implies $(D_{\vp}\Phi)_{(\hat{\vp},\hat{t})}$ is surjective.

\subsection{Closeness}
Since $0\in I$ and $I$ is open, there exists $t_{0}\in (0,1]$ such that $[0,t_{0})\subset I$. We need to prove $t_{0}\in I$. It suffices to prove the following proposition.

\begin{proposition}\label{A priori estimate}
Let $\vp_{t}$ be the solution of (\ref{Fu-Yau equation t}). If $\vp_{t}$ satisfies (\ref{Elliptic condition t}) and (\ref{Normalization condition t}), there exists a constant $C_{A}$ depending only on $A$, $t_{0}$, $\rho$, $\mu$, $\alpha$, $\beta$ and $(M,\omega)$ such that
\begin{equation*}
\|\vp_{t}\|_{C^{2,\beta}}\leq C_{A}.
\end{equation*}
\end{proposition}

\begin{proof} First, we  prove the zero order estimate.   In fact,  we have
\begin{claim}\label{claim-2}
\begin{equation}\label{Closeness equation 1}
\sup_{M}e^{-\vp_{t}} \leq 2M_{0}A, ~t\in[0,t_{0}),
\end{equation}
where $M_{0}$ is the constant in Proposition \ref{Zero order estimate}.
\end{claim}
 Note that  $\vp_{0}=-\ln A$.  Then $\sup_{M}e^{-\vp_{0}} \leq M_{0}A$, which satisfies (\ref{Closeness equation 1}).   Thus,  if (\ref{Closeness equation 1}) is false,  there will  exist  $\tilde{t}\in (0,t_{0})$ such that
\begin{equation}\label{Closeness equation 2}
\sup_{M}e^{-\vp_{\tilde{t}}} = 2M_{0}A.
\end{equation}
We may assume  that $2M_{0}A\le   \delta_{0}$,  where $\delta_{0}=\sqrt{\frac{1}{2|\alpha|\|\rho\|_{C^{0}}+1}}$ is chosen  as in Proposition \ref{Zero order estimate}.   Namely, $e^{-\vp_{\ti{t}}}\leq\delta_{0}$. Hence,  we can  apply  Proposition \ref{Zero order estimate} to $\vp_{\ti{t}}$  whlie  $\rho$ and  $\mu$  are replaced by $t\rho$ and   $t\mu$, respectively,  and  we obtain
\begin{equation*}
e^{-\vp_{\tilde{t}}} \leq M_{0}A,
\end{equation*}
which contradicts to (\ref{Closeness equation 2}). This proves (\ref{Closeness equation 1}). Combining (\ref{Closeness equation 1}) and Proposition \ref{Zero order estimate}, we obtain the zero order estimate

Next, we use the similar argument to prove the second order estimate
\begin{equation}\label{Closeness equation 4}
\sup_{M}|\de\dbar\vp_{t}|_{g} \leq D_{0},
\end{equation}
for any $t\in(0,t_{0})$, where $D_{0}$ is the constant as in Proposition \ref{Second order estimate}. If (\ref{Closeness equation 4}) is false, there exists  $\tilde{t}\in (0,t_{0})$ such that
\begin{equation*}
\sup_{M}|\de\dbar\vp_{\ti{t}}|_{g} = D_{0}.
\end{equation*}
Recalling Proposition \ref{Second order estimate}, we get
\begin{equation*}
\sup_{M}|\de\dbar\vp_{\ti{t}}|_{g} \leq \frac{D_{0}}{2},
\end{equation*}
which is a contradiction. Thus (\ref{Closeness equation 4}) is true.

By (\ref{Closeness equation 4}) and Proposition \ref{Gradient estimate}, we have  the first order estimate
\begin{equation}\label{Closeness equation 3}
\sup_{M}|\de\vp_{t}|_{g}^{2} \leq C.
\end{equation}
Combining  (\ref{Closeness equation 4}) and (\ref{Closeness equation 3}) with equation (\ref{Fu-Yau equation t}) (Note that (\ref{Fu-Yau equation t}) is equivalent to a $2$-nd Hessian type equation as (\ref{fu-2})),  we get
\begin{equation}\label{non-degenrate}
\left|\sigma_{2}(\ti{\omega})-\frac{n(n-1)}{2}e^{2\vp}\right| \leq Ce^{\vp}.
\end{equation}
Then,  by the zero order estimate, we deduce
\begin{equation*}
\frac{1}{CA^{2}} \leq \sigma_{2}(\ti{\omega}) \leq \frac{C}{A^{2}}.
\end{equation*}
Hence,   (\ref{Fu-Yau equation t})  is uniformly elliptic and non-degenerate.  By  the $C^{2,\alpha}$-estimate (cf. \cite[Theorem 1.1]{TWWY15}), we obtain
\begin{equation}\label{c2-alpha}
\|\vp_{t}\|_{C^{2,\beta}}\leq C_{A}.
\end{equation}

\end{proof}

\subsection{Uniqueness}
In this subsection, we give the proof of Theorem \ref{Uniqueness Theorem}. First, we show the uniqueness of solutions to (\ref{Fu-Yau equation t}) when $t=0$.

\begin{lemma}\label{Uniqueness lemma 1}
When $t=0$, (\ref{Fu-Yau equation t}) has a unique solution
\begin{equation*}
\vp_{0} = -\ln A.
\end{equation*}
\end{lemma}

\begin{proof}
By the similar calculation of (\ref{Infimum estimate equation 3}) (taking $k=1$), we obtain
\begin{equation}
\begin{split}
& \int_{M}e^{-\vp}(e^{\vp}\omega+\alpha e^{-\vp}t\rho)\wedge\sqrt{-1}\de\vp\wedge\dbar\vp\wedge\omega^{n-2} \\
\leq {} & 2\alpha \int_{M}e^{-2\vp}\sqrt{-1}\de\vp \wedge\dbar(t\rho)\wedge\omega^{n-2}-2\int_{M}e^{-\vp}t\mu\frac{\omega^{n}}{n!}.\notag
\end{split}
\end{equation}
When $t=0$, it is clear that
\begin{equation*}
\int_{M}\sqrt{-1}\de\vp\wedge\dbar\vp\wedge\omega^{n-1} = 0.
\end{equation*}
Combining this with the normalization condition $\|e^{-\vp}\|_{L^{1}}=A$, we obtain
\begin{equation*}
\vp_{0} = -\ln A.
\end{equation*}
\end{proof}

\begin{proof}[Proof of Theorem \ref{Uniqueness Theorem}]
Assume that we have two solutions $\vp$ and $\vp'$ of  (\ref{Fu-Yau equation}).   We use the continuity method to solve (\ref{Fu-Yau equation t}) from $t=1$ to $0$. Note that $\vp $ and $\vp'$ are both solutions when $t=1$.   Then by  the implicit function theorem
as in  Subsection 5.1,    there is a smooth solution $\varphi_t^1$  (or  $\varphi_t^2$) of  (\ref{Fu-Yau equation t})  for  any $t\in (t_0, 1]$   ($t_0<1$) with the property $\varphi_1^1= \vp $ (or  $\varphi_1^2=\vp'$).
Set
\begin{align}
J_{\vp }= \{ t\in [0,1] ~|~&\text{there exists  a  family of smooth solutions}~ \varphi_{t'}^1~{\rm of}~ (\ref{Fu-Yau equation t} )\notag\\
& {\rm for ~any}~t'\in [t,1]~ \text{such that} ~ \vp_{1}^{1} = \vp  \}.\notag
\end{align}
From the argument in Section 2-4,  we see that Proposition  \ref{Zero order estimate},  Proposition \ref{Gradient estimate}  and Proposition \ref{Second order estimate} are  still true for  $\varphi_t^1$. As a consequence,  Proposition \ref{A priori estimate}  holds  for  $\varphi_t^1$.  Thus  $J_{\vp }=  [0,1]$. Similarly,  $J_{\vp' }=  [0,1]$. On the other hand, thanks to Lemma \ref{Uniqueness lemma 1},  we have
\begin{equation}
\vp_{0}^1 = \vp_{0}^2 = -\ln A.\notag
\end{equation}
Hence $\varphi_t^1=\varphi_t^2$  for any $t\in  [0,1]$. Theorem \ref{Uniqueness Theorem} is proved.
\end{proof}

It seems that the condition (\ref{condition-uniqueness}) in   Theorem \ref{Uniqueness Theorem}  can be removed.    In precise,  we have the following conjecture.

\begin{conjecture}\label{uniqueness-general}
The solution $\vp$ of (\ref{Fu-Yau equation}) in  Theorem \ref{Existence Theorem} is unique.
\end{conjecture}

\begin{remark}We remark  that conjecture \ref{uniqueness-general}  is true if  $\alpha<0$ and $\rho\ge 0$ in  equation (\ref{Fu-Yau equation}).  In fact,  by modifying  the argument in the proof of Proposition \ref{Zero order estimate}, we can get the  $C^0$-estimate for the  solution  $\vp_t$  of equation  (\ref{Fu-Yau equation t})  by the assumption of  (\ref{Elliptic condition}) and  (\ref{Normalization condition}) in this case. Then  by  the $C^2$-estimate in  \cite[Proposition 5, Proposition 6]{PPZ16b}, we  can also obtain  (\ref{c2-alpha}). We will discuss it  for details somewhere.

\end{remark}

\end{document}